\documentclass[12pt, one side,emlines]{amsart}
\usepackage[a4paper,margin=1in]{geometry}
\usepackage{fancyhdr}
\usepackage{enumerate}
\usepackage{amssymb,latexsym,xy,eucal,mathrsfs}
\textwidth=18cm \textheight=27cm \theoremstyle{plain}
\usepackage{tikz}
\usetikzlibrary{positioning}
\theoremstyle{definition}

\newtheorem{theorem}{Theorem}[section]
\newtheorem{thm}[theorem]{Theorem}
\newtheorem{lem}[theorem]{Lemma}

\newtheorem{corollary}[theorem]{Corollary}

\newtheorem{defn}{Definition}[section]

\begin{document}
	

	%
	\title{Elementary Lift and Single Element Coextension of A Binary Gammoid}\maketitle
	
	\markboth{ Shital Dilip Solanki, Ganesh Mundhe and S. B. Dhotre}{ Elementary Lift and Single Element Coextension of A Binary Gammoid}\begin{center}\begin{large} Shital Dilip Solanki$^1$, Ganesh Mundhe$^2$ and S. B. Dhotre$^3$ \end{large}\\\begin{small}\vskip.1in\emph{
				1. Ajeenkya DY Patil University, Pune-411047, Maharashtra,
				India\\ 
				2. Army Institute of Technology, Pune-411015, Maharashtra,
				India\\
				3. Department of Mathematics,
				Savitribai Phule Pune University,\\ Pune - 411007, Maharashtra,
				India}\\
			E-mail: \texttt{1. shital.solanki@adypu.edu.in, 2. gmundhe@aitpune.edu.in, 3. dsantosh2@yahoo.co.in. }\end{small}\end{center}\vskip.2in
	\begin{abstract} Elementary lift is a generalized splitting operation. Splitting of a binary gammoid need not be a gammoid. This paper finds forbidden minors for a binary gammoid whose elementary lift is a binary gammoid. We also find forbidden minors for a binary gammoid whose single element coextension is a binary gammoid.	\end{abstract}\vskip.2in
	\noindent\begin{Small}\textbf{Mathematics Subject Classification (2010)}:
		05B35,05C83, 05C50     \end{Small}\vskip.2in
	\vskip.25in

	\baselineskip 19truept 
	\section{Introduction}
\noindent
Initially, the splitting operation in graphs is introduced by Fleischner \cite{fl} as follows.\\
Let $a=(v_1,v)$ and $b=(v_2,v)$ be two arcs incident at node $v$. Then remove arcs $a$, $b$ and adding new node $v'$ and arcs $a=(v_1,v')$ and $b=(v_2,v')$. This operation is called splitting using two arcs. This operation is extended to binary matroid by Raghunathan et al. \cite{ttr} as a splitting of binary matroid using two element set. Later, this operation is generalized as a splitting of matroid using n-element set as follows, by Shikare et al. \cite{mms1}.  

\begin{defn}\cite{mms1}
	Let a binary matroid $B$ with $H \subseteq E(B)$ and let $A$ be the matrix representing the binary matroid $B$. Obtain a new matrix $A_H$ by adding a new row at the bottom of the matrix $A$ with entries 0 everywhere except 1 at the columns representing elements of $H$. Then the vector matroid of $A_H$ is called splitting matroid $B_H$, and the operation is called splitting using a set.
\end{defn}  
\noindent Later, for a binary matroid, Azadi \cite{azt} defined the element splitting operation as follows.
\begin{defn}\cite{azt} 
	Let a binary matroid is $B$ and $H \subseteq E(B)$. Let the matrix $A$ represent binary matroid $B$. Obtain a matrix $A_H'$ by adding a row at the bottom of the matrix $A$ with entries $0$ everywhere except entries $1$ in the columns corresponding to the elements in $H$ and adding one column with entries zero except in the last row where the entry is $1$. Then the vector matroid corresponding to the matrix $A_H'$ is the element splitting matroid $B_H'$, and the operation is known as the element splitting operation.
\end{defn}
\noindent Later, for binary matroid, Azanchilar \cite{azn} defined the es-splitting operation as follows.
\begin{defn}
Let the matrix $A$ represents a binary matroid $B$ and $H \subseteq E(B)$ with $e \in H$. Let $N$ be a matrix obtained from $A$ by adding a row labeled $\gamma$, similar to a row labeled by $e$. Let $M$ be the vector matroid of the matrix $N$, then the element splitting matroid $M_H'$ of $M$ is called es-splitting matroid of $B$, and $ B_H^e$ denotes it, and the operation is called es-splitting operation.  
\end{defn}
\noindent For undefined concepts, we refer to Oxley \cite{ox}.\\
Let $M$ and $Q$ be matroids and let matroid $N$ be such that for some subset $X$ of $E(N)$, $M=N\backslash X$ and $Q=N/X$. Then $Q$ is called the quotient of $M$, and $M$ is called the lift of $Q$. If $X$ is a singleton set, then $Q$ is called the elementary quotient of $M$ and $M$ is called the elementary lift of $Q$. From the definition of splitting and coextension, it is observed that the elementary lift of $Q$ is the splitting matroid $Q_T$ for some $T$. From the definition of element splitting operation, it is clear that the single element coextension of the given matroid $B$ is the element splitting matroid $B_T'$ for some $T\subseteq E(B)$.

\noindent It was observed that the elementary lift or splitting operation and single element coextension or element splitting operation of binary matroid do not preserve the properties like graphicness, cographicness, contentedness etc. Borse \cite{ymb1} characterized graphic matroid, which gives graphic matroid under splitting operation using two-element set. Similarly, graphic matroid is characterized by Mundhe \cite{gm} whose elementary lift is graphic.   \\
Thus, in the same way, it is very interesting to characterize the class of  binary gammoid whose elementary lift and single element coextension is gammoid.  

\noindent Characterization of a binary gammoid is given as follows.
\begin{thm}\label{gammoid}\cite{ox}
	For a matroid $M$ the following are equivalent.\\
	(i) $M$ is a binary gammoid.\\
	(ii)$M$ is a graphic gammoid. \\
	(iii) $M$ is a regular gammoid. \\
	(iv) $M$ has no minor isomorphic to $U_{2,4}$ or $M(K_4)$.
\end{thm}

\noindent{\bf Notation:} Let $\mathcal{G}_k$ denote the class of a binary gammoid which do not yield a binary gammoid after splitting using the k-element set. 

Borse \cite{ymb_Gammoid} characterized a binary gammoid whose splitting and element splitting is a binary gammoid as follows.

\begin{thm}\cite{ymb_Gammoid}\label{G2}
	Let  a binary gammoid be $M$, then $M \in \mathcal{G}_2$ if and only if $M$ contain $M(G_1)$ minor, where $G_1$ is shown in Figure \ref{G2fig}.
\end{thm}
\begin{figure}[h!]
	\centering
	\unitlength 1mm 
	\linethickness{0.4pt}
	\ifx\plotpoint\undefined\newsavebox{\plotpoint}\fi 
	\begin{picture}(29.56,28)(0,0)
		\put(28.81,11.25){\circle*{1.5}}
		\put(9.31,11.281){\circle*{1.5}}
		\put(9.158,11.299){\line(1,0){19.534}}
		\put(9.473,27){\circle*{1.5}}
		\put(28.79,27.25){\circle*{1.5}}
		\put(9.358,27.102){\line(1,0){19.325}}
		\put(28.683,27.102){\line(0,-1){15.608}}
		\put(9.21,11.196){\line(0,1){16.352}}
		\multiput(28.75,27.5)(-.0409751037,-.0337136929){482}{\line(-1,0){.0409751037}}
		\qbezier(9,11.5)(2.5,18.25)(9,27)
		\qbezier(9,11.5)(18.125,2)(28.75,11.5)
		\put(3,19){\makebox(0,0)[cc]{$x$}}
		\put(17.25,4.5){\makebox(0,0)[cc]{$y$}}
		\put(16.75,0){\makebox(0,0)[cc]{$G_1$}}
	\end{picture}
	\caption{Minimal minor of the class $\mathcal{G}_2$.}
	\label{G2fig}
\end{figure}
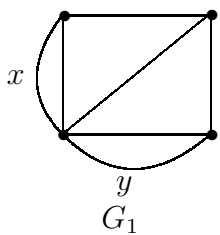

\begin{thm}\cite{ymb_Gammoid}
	Let $M$ be a binary gammoid, then element splitting matroid $M_H'$ is a binary gammoid, for any $H \subseteq E(M)$ with $|H|=2$ if and only if $M$ does not contain $M(G_6)$ minor, where $G_6$ is as shown in Figure \ref{elementsp1}.
\end{thm}
\begin{figure}[h!]
	\unitlength 1mm 
	\linethickness{0.4pt}
	\ifx\plotpoint\undefined\newsavebox{\plotpoint}\fi 
	\begin{picture}(28.264,24.188)(0,0)
		\put(26.81,8.438){\circle*{1.5}}
		\put(16.81,23.438){\circle*{1.5}}
		\put(7.31,8.469){\circle*{1.5}}
		\put(7.158,8.487){\line(1,0){19.534}}
		\multiput(16.595,23.588)(.0336731392,-.0485889968){309}{\line(0,-1){.0485889968}}
		\multiput(16.743,23.291)(-.0336655052,-.0512787456){287}{\line(0,-1){.0512787456}}
		\qbezier(16.892,23.44)(28.264,21.135)(26.852,8.426)
		\qbezier(7.23,8.574)(5,20.764)(16.446,23.44)
		\put(17.25,3.25){\makebox(0,0)[cc]{$G_6$}}
	\end{picture}
	\caption{Forbidden minor for element splitting of a binary gammoid.}
	\label{elementsp1}
\end{figure}
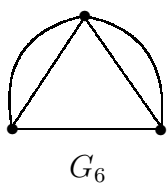

\noindent 
This paper aims to characterize a binary gammoid whose splitting and element splitting is a binary gammoid. We prove the main theorems as stated below.
\begin{thm}\label{mt1}
	Let a binary gammoid be $M$ and $k \geq 2$ such that $M \in \mathcal{G}_k$. Then $M$ contain a minor $P$, for which one of the following holds.\\
	(i) $P$ is isomorphic to the one element extension of some minimal minor of the class $\mathcal{G}_{k-1}$.\\
	(ii) $P = M(Q_i)$ or extension of $M(Q_i)$ not more than k elements, where the graph $Q_i$ is shown in Figure \ref{qk4fig}, for $i=2,3,4$.
\end{thm}

\begin{thm} \label{mt2}
	Let a binary gammoid be $M$ then $M \notin  \mathcal{G}_3$ if and only if $M$ do not contain $M(G_i)$ minor, where the graph $G_i$ is shown in Figure \ref{G3}, for $i=2,3,4$. 
\end{thm}
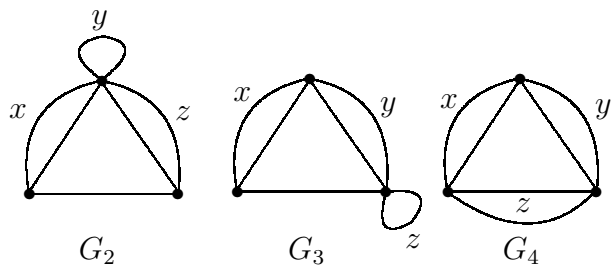
\begin{figure}[h!]
	\centering
\unitlength 1mm 
\linethickness{0.4pt}
\ifx\plotpoint\undefined\newsavebox{\plotpoint}\fi 
\begin{picture}(82.015,31.5)(0,0)
	\put(25.56,8.438){\circle*{1.5}}
	\put(15.56,23.438){\circle*{1.5}}
	\put(6.06,8.469){\circle*{1.5}}
	\put(5.908,8.487){\line(1,0){19.534}}
	\multiput(15.345,23.588)(.0336731392,-.0485889968){309}{\line(0,-1){.0485889968}}
	\multiput(15.493,23.291)(-.0336655052,-.0512787456){287}{\line(0,-1){.0512787456}}
	\qbezier(15.642,23.44)(27.014,21.135)(25.602,8.426)
	\qbezier(5.98,8.574)(3.75,20.764)(15.196,23.44)
	\qbezier(15.345,23.737)(9.547,28.048)(15.642,29.386)
	\qbezier(15.642,29.386)(21.662,28.419)(15.493,23.588)
	\put(52.912,8.73){\circle*{1.5}}
	\put(42.912,23.73){\circle*{1.5}}
	\put(33.412,8.76){\circle*{1.5}}
	\put(33.26,8.779){\line(1,0){19.534}}
	\multiput(42.697,23.88)(.0336731392,-.0485889968){309}{\line(0,-1){.0485889968}}
	\multiput(42.845,23.583)(-.0336655052,-.0512787456){287}{\line(0,-1){.0512787456}}
	\qbezier(42.994,23.731)(54.366,21.427)(52.954,8.717)
	\qbezier(33.332,8.866)(31.102,21.055)(42.548,23.731)
	\qbezier(52.953,8.723)(59.94,9.541)(56.224,4.71)
	\qbezier(56.224,4.71)(51.021,1.811)(52.953,8.723)
	\put(80.561,8.73){\circle*{1.5}}
	\put(70.561,23.73){\circle*{1.5}}
	\put(61.061,8.76){\circle*{1.5}}
	\put(60.909,8.779){\line(1,0){19.534}}
	\multiput(70.346,23.88)(.0336731392,-.0485889968){309}{\line(0,-1){.0485889968}}
	\multiput(70.494,23.583)(-.0336655052,-.0512787456){287}{\line(0,-1){.0512787456}}
	\qbezier(70.643,23.731)(82.015,21.427)(80.603,8.717)
	\qbezier(60.981,8.866)(58.751,21.055)(70.197,23.731)
	\qbezier(60.981,8.723)(72.575,.25)(80.603,9.02)
	\put(15,1){\makebox(0,0)[cc]{$G_2$}}
	\put(42.5,1){\makebox(0,0)[cc]{$G_3$}}
	\put(70.75,1){\makebox(0,0)[cc]{$G_4$}}
	\put(4.5,19.5){\makebox(0,0)[cc]{$x$}}
	\put(15.25,31.5){\makebox(0,0)[cc]{$y$}}
	\put(26.25,19.25){\makebox(0,0)[cc]{$z$}}
	\put(34,21.75){\makebox(0,0)[cc]{$x$}}
	\put(53.25,20){\makebox(0,0)[cc]{$y$}}
	\put(56.5,2.25){\makebox(0,0)[cc]{$z$}}
	\put(61.25,20.75){\makebox(0,0)[cc]{$x$}}
	\put(81.5,19.75){\makebox(0,0)[cc]{$y$}}
	\put(71,7){\makebox(0,0)[cc]{$z$}}
\end{picture}

	\caption{Minimal minor of the class $\mathcal{G}_3$}
	\label{G3}
\end{figure}

\begin{thm}
	Let a binary gammoid be $M$, then element splitting matroid $M_H'$ is a binary gammoid for any $H \subseteq E(M)$, with $|H| \geq 3$ if and only if $M$ do not contain minor $M(G_6)$, where $G_6$ is as shown in Figure-\ref{elementsp1}.
\end{thm}

\noindent This paper also characterizes a binary gammoid whose es-splitting is a binary gammoid. We proved the theorem stated as below.
\begin{thm}
	Let $M$ be a binary gammoid then es-splitting matroid $M_T^e$ is a binary gammoid, for $T \subseteq E(M)$, with $|T|\geq 2$ and $e \in T$ if and only if $M$ do not contain a minor $M(G_7)$, where $G_7$ is as shown in Figure \ref{es-sp}.
\end{thm}
\begin{figure}[h]
	\centering
	\unitlength 1mm 
	\linethickness{0.4pt}
	\ifx\plotpoint\undefined\newsavebox{\plotpoint}\fi 
	\begin{picture}(27.81,26.688)(0,0)
		\put(27.06,10.938){\circle*{1.5}}
		\put(17.06,25.938){\circle*{1.5}}
		\put(7.56,10.969){\circle*{1.5}}
		\put(7.408,10.987){\line(1,0){19.534}}
		\multiput(16.845,26.088)(.0336731392,-.0485889968){309}{\line(0,-1){.0485889968}}
		\multiput(16.993,25.791)(-.0336655052,-.0512787456){287}{\line(0,-1){.0512787456}}
		\put(15.5,.75){\makebox(0,0)[cc]{$G_7$}}
		\qbezier(26.75,11)(17.5,.5)(7.25,11)
		\put(9.75,20.25){\makebox(0,0)[cc]{$x$}}
		\put(17,13){\makebox(0,0)[cc]{$y$}}
	\end{picture}
	
	\caption{Forbidden minor for es-splitting of a binary gammoid}
	\label{es-sp}
\end{figure}
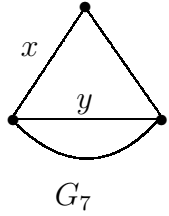
\section{Preliminary Results}	
For $k \geq 2 $, $\mathcal{G}_k$ denote a collection of a binary gammoid whose splitting is not a binary gammoid using a set with k elements. In this paper, we aim to characterize minimal minors that belong to the class $\mathcal{G}_k$, for $k \geq 2$. 

\begin{lem}\label{minlem}
	Let $M$ be a binary gammoid such that $M_H$ has $M(K_4)$ minor for some $H \subseteq E(M)$ with $|H|\geq 2$. Then $M$ contains a minor $P$ with $H \subseteq E(P)$, such that $P$ satisfies one of the following conditions. 
	\begin{enumerate}[(i).]
		\item $P_H \cong M(K_4)$;
		\item $P_H/H' \cong M(K_4)$, for some $H'\subseteq H$;
		\item $P$ is a one element extension of some minimal minor in the class $\mathcal{G}_{k-1}$, for $k \geq 2$.
	\end{enumerate}  
\end{lem}
\begin{proof}
Let a binary gammoid $M$ be such that $M_H$ contains $M(K_4)$ minor. Thus there exists subsets $H_1$ and $H_2$ of $E(M)$ such that $M_H \backslash H_1 /H_2 \cong M(K_4)$. Let $H_i'=H \cap H_i$ and $H_i''=H_i-H_i'$, for $i=1,2$. Then $M_H\backslash H_1''/H_2'' \cong (M\backslash H_1''/H_2'')_H$ as each $H_i''$ is disjoint from $H$. Let $P=M\backslash H_1''/H_2''$. Then $P$ is a minor of $M$ containing $H$. Consider $P_H\backslash H_1'/H_2' \cong (M\backslash H_1''/H_2'')_H \backslash H_1'/H_2'\cong M_H\backslash H_1''/H_2''\backslash H_1'/H_2'\cong M_H \backslash H_1'\cup H_1''/H_2'\cup H_2'' \cong M_H\backslash H_1/H_2 \cong M(K_4)$. That is $P_H\backslash H_1'/H_2' \cong M(K_4)$.\\
If $H_1' = H_2'=\emptyset $ Then (i) holds. \\
If $H_1'=\emptyset$ and $H_2' \neq \emptyset$ Then (ii) holds.\\
If $H_1' \neq \emptyset$. $H_1' \subseteq H$ then $|H_1'| \leq k$. Now, if $|H_1'|= k$, then $H_1'=H$ and $H_2'=\emptyset$ it follows that $M(K_4) \cong P_H\backslash H_1' \cong P_H\backslash H \cong P\backslash H$. Thus, a binary gammoid $M$ contains $M(K_4)$ minor, a contradiction. Hence $|H_1'|\neq k$. If $0 < |H_1'| < k$. Let $h \in H_1'$, $T=H-{h}$ and $T'=H_1'-{h}$. Assume that $N=P\backslash h$ then $N$ is a minor of $P$ and hence it is a minor of a binary gammoid $M$, thus $N$ is a binary gammoid and $|T|=k-1$. $P_H\backslash h \cong (P\backslash h)_{(H-{h})} \cong N_T$. As $P_H\backslash H_1'/H_2' \cong M(K_4)$ implies that $N_T\backslash T'/H_2' \cong M(K_4)$, thus $N$ is the minimal minor in the class $\mathcal{G}_{k-1}$. As $N=P\backslash h$, $P$ is the extension of $N$ by one element.  
\end{proof}

\begin{lem}\label{rel_min_quo}
Let a binary gammoid $P$ be as stated in Lemma \ref{minlem}. There is a binary gammoid $Q$ with $a \in E(Q)$, such that $Q\backslash a \cong M(K_4)$ and $P \cong Q/a$ or $P$ is coextension of $Q/a$ by not more than k elements. 	
\end{lem}
\begin{proof}
Let $H \subseteq E(P)$ and $P$ be a binary gammoid such that $P_H \cong M(K_4)$ or $P_H/H' \cong M(K_4)$ for some subset $H'$ of $H$. From the definition of splitting and element splitting operation, we can easily verify the relations, $P_H'\backslash a \cong P_H$ and $P_H'/a \cong P$ where $P_H'$ is element splitting of the binary gammoid $P$ using a set $H$. \\ 
Case-(i). If $P_H \cong M(K_4)$, We assume that $Q=P_H'$ where $E(Q)= E(P)\cup a$. Then $Q/a \cong P_H'/a \cong P$ and $P_H'\backslash a \cong Q\backslash a\cong P_H \cong M(K_4)$. Thus $P \cong Q/a$ \\
Case-(ii). If  $P_H/H' \cong M(K_4)$ then we assume that $Q = P_H'/H'$ thus $ Q \backslash a \cong P_H'/H'\backslash a \cong P_H'\backslash a /H' \cong P_H/H'\cong M(K_4)$. Hence $Q/a$ is quotient of $M(K_4)$. Also $Q/a = P_H'/H'/a \cong P_H'/a/H' \cong P/H'$. Hence $P$ is coextension of $Q/a$ by $H'$ and $|H'| \leq k$. Thus we say that $P$ is coextension of $Q/a$ by not more than k elements.
\end{proof}
\noindent For a binary matroid $M$, if $M \cong Q\backslash a$ for some matroid $Q$ with $a \in E(Q)$,  then $Q/a$ is called quotient of $M$ and $M$ is called elementary lift of $Q/a$. Thus to find minimal minor of the class $\mathcal{G}_k$, we need to find minimal minor $P$ of a binary gammoid $M$,  such that $P_H \cong M(K_4)$ or $P_H/H' \cong M(K_4)$ for some subset $H'$ of $H$, where $H\subseteq (M)$. Thus by Lemma \ref{rel_min_quo}, $P\cong Q/a$ or $P$ is extension of $Q/a$ not more than k elements. Thus we need to find $Q/a$, that is quotient of $M(K_4)$. Hence, in the following lemma, we find graphic quotient of $M(K_4)$. 
\begin{lem} \label{qk4}
Every graphic quotient of $M(K_4)$ is isomorphic to one of the matroid $M(Q_i)$ where $Q_i$ is shown in Figure \ref{qk4fig}, for $i=1,2,3,4$. 
\end{lem}	
\begin{figure}[h!]
	\centering
\unitlength 1mm 
\linethickness{0.4pt}
\ifx\plotpoint\undefined\newsavebox{\plotpoint}\fi 
\begin{picture}(105.583,31.001)(0,0)
	\put(49.129,10.053){\circle*{1.5}}
	\put(39.129,25.053){\circle*{1.5}}
	\put(29.629,10.084){\circle*{1.5}}
	\put(29.476,10.102){\line(1,0){19.534}}
	\multiput(38.913,25.203)(.0336763754,-.0485889968){309}{\line(0,-1){.0485889968}}
	\multiput(39.062,24.906)(-.0336655052,-.0512787456){287}{\line(0,-1){.0512787456}}
	\qbezier(39.211,25.055)(50.582,22.75)(49.17,10.041)
	\qbezier(29.548,10.189)(27.319,22.379)(38.765,25.055)
	\qbezier(38.913,25.352)(33.116,29.663)(39.211,31.001)
	\qbezier(39.211,31.001)(45.231,30.034)(39.062,25.203)
	\put(76.48,10.345){\circle*{1.5}}
	\put(66.48,25.345){\circle*{1.5}}
	\put(56.98,10.375){\circle*{1.5}}
	\put(56.828,10.394){\line(1,0){19.534}}
	\multiput(66.265,25.495)(.0336731392,-.0485889968){309}{\line(0,-1){.0485889968}}
	\multiput(66.413,25.198)(-.0336655052,-.0512787456){287}{\line(0,-1){.0512787456}}
	\qbezier(66.562,25.346)(77.934,23.042)(76.522,10.332)
	\qbezier(56.9,10.481)(54.67,22.67)(66.116,25.346)
	\qbezier(76.522,10.338)(83.509,11.156)(79.792,6.325)
	\qbezier(79.792,6.325)(74.59,3.426)(76.522,10.338)
	\put(104.129,10.345){\circle*{1.5}}
	\put(94.129,25.345){\circle*{1.5}}
	\put(84.629,10.375){\circle*{1.5}}
	\put(84.477,10.394){\line(1,0){19.534}}
	\multiput(93.914,25.495)(.0336731392,-.0485889968){309}{\line(0,-1){.0485889968}}
	\multiput(94.062,25.198)(-.0336655052,-.0512787456){287}{\line(0,-1){.0512787456}}
	\qbezier(94.211,25.346)(105.583,23.042)(104.171,10.332)
	\qbezier(84.549,10.481)(82.319,22.67)(93.765,25.346)
	\qbezier(84.549,10.338)(96.144,1.865)(104.171,10.635)
	\put(22.38,10.007){\circle*{1.35}}
	\put(4.83,10.035){\circle*{1.35}}
	\put(4.693,10.051){\line(1,0){17.581}}
	\put(4.976,24.182){\circle*{1.35}}
	\put(22.361,24.407){\circle*{1.35}}
	\put(4.873,24.274){\line(1,0){17.392}}
	\put(22.265,24.274){\line(0,-1){14.047}}
	\put(4.74,9.958){\line(0,1){14.717}}
	\multiput(22.326,24.632)(-.0409585253,-.0336981567){434}{\line(-1,0){.0409585253}}
	\multiput(4.835,24.386)(.0405783837,-.0337338371){430}{\line(1,0){.0405783837}}
	\put(13.454,2){\makebox(0,0)[cc]{$Q_1$}}
	\put(39.102,2){\makebox(0,0)[cc]{$Q_2$}}
	\put(66.221,2){\makebox(0,0)[cc]{$Q_3$}}
	\put(95.022,2){\makebox(0,0)[cc]{$Q_4$}}
\end{picture}

\caption{Quotient for $K_4$}
\label{qk4fig}
\end{figure}
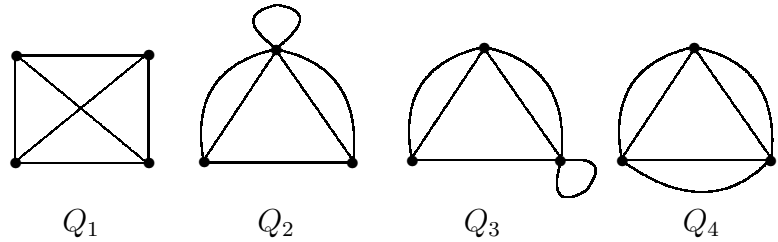
\begin{proof}
	Let $N$ be binary matroid with $x \in E(N)$ such that $N \backslash x \cong M(K_4)$ and $N /x$ be a binary gammoid and hence graphic by Lemma \ref{gammoid}. Then for some connected graph $G$, $N/x \cong M(G)$. If ${x}$ is cocircuit or circuit of $N$ then $N/x \cong N \backslash x \cong M(K_4)$. Thus $G \cong Q_1$.
	
	If ${a}$ is not a circuit or a cocircuit of $N$. Then $N \backslash x $ contains 4 nodes and 6 arcs. Then $N$ contains 4 nodes and 7 arcs. Thus $r(N/x)=2$ and $E(N/x)=6$. Thus graph $G$ contains 3 nodes and  6 arcs. $G$ cannot be simple graph as there is no simple graph on 3 nodes and 6 arcs.\\
	If $G$ contains more than one loop, then $M(G)\cong N/x$ contains more than one loop, then $N\backslash x \cong M(K_4)$ will contain a loop or 2-circuit, a contradiction. Hence $G$ does not contains more than one loop. Also, If $G$ contains more than two arcs having the same end nodes. Then $M(G) \cong N/x$ contains more than two parallel elements; thus $N\backslash x\cong M(K_4)$ will contain 2-circuit, a contradiction. Thus $G$ contains not more than two parallel arcs having the same end nodes. Hence there are two cases. \\
Case-(i). If $G$ contains one loop, then $G$ can be obtained by adding a loop to a graph with 3 nodes and 5 arcs. By Harary \cite{Har} (page 226), there is only one graph on 3 nodes and 5 arcs. Hence $G\cong Q_2$ and $G\cong Q_3$.\\
Case-(ii). Suppose $G$ does not contain a loop. Then $G$ contains three pairs of parallel arcs. By Harary \cite{Har} (page 226), there is only one graph on 3 nodes and 6 arcs. Hence $G\cong Q_4$.
\end{proof}

\section{Elementary Lifts of Binary Gammoid}
In this section, we characterize a binary gammoid whose elementary lift is a binary gammoid. We state the theorem again.

\begin{thm}\label{mt1}
		Let a binary gammoid be $M$ and $k \geq 2$ such that $M \in \mathcal{G}_k$. Then $M$ contain a minor $P$ for which one of the following holds.\\
	(i) $P$ is isomorphic to the single element extension of some minimal minor of the class $\mathcal{G}_{k-1}$.\\
	(ii) $P = M(Q_i)$ or extension of $M(Q_i)$ by at most k elements, where the graph $Q_i$ is shown in Figure \ref{qk4fig}, for $i=2,3,4$.
\end{thm}
\begin{proof}
	Let a binary gammoid $M$ be such that $M \in \mathcal{G}_k$. Thus, by Lemma \ref{gammoid}, $M_H$ contain a minor $M(K_4)$ for some subset $H$ of $E(M)$ with $|H|\geq 2$. Then by Lemma \ref{minlem}, $M$ contains a minor $P$ containing $H$ such that one of the following holds.\\
	(i). $P_H \cong M(K_4)$ or $P_H/H'\cong M(K_4)$ for some $H' \subseteq H$;\\
	(ii). $P$ is a one element extension of some minimal minor of the class $\mathcal{G}_{k-1}$.\\
	Condition (i) follows immediately from Lemma \ref{minlem}. Now, if  $P_H \cong M(K_4)$ or $P_H/H'\cong M(K_4)$ for some $H' \subseteq H$ then by Lemma \ref{rel_min_quo} there is a minor $Q$ with $a \in E(Q)$ such that $Q \backslash a \cong M(K_4)$ and $P= Q/a$ or extension of $Q/a$ by at most k elements. \\
	By Lemma \ref{qk4}, it is proved that, when $Q\backslash a \cong M(K_4)$ then $Q/a \cong M(Q_i)$, where the graph $Q_i$ is as shown in Figure \ref{qk4fig}, for i=1,2,3,4. If $P=M(Q_1)$ and $P$ is a minor of a binary gammoid, a contradiction. Hence, either $P=M(Q_i)$ or P is an extension of $M(Q_i)$ not more than k elements, for i=2,3,4.
\end{proof}
\begin{corollary}\label{maincor}
	Let $M$ be a binary gammoid and $H \subseteq E(M)$, with $|H|\geq2$. Then $M_H$ is a binary gammoid if and only if $M$ does not contain a minor $M(Q_i)$, for $i=2,3,4$.
\end{corollary}
\begin{proof}
If $M$ contain minor $M(Q_i)$, then easy to prove that $M_H$ is not a binary gammoid,for $i=2,3,4$.\\
	Conversely, let $M$ be such that it does not contain $M(Q_i)$ minor, for $i=2,3,4$. On contrary, assume that $M_H$ is not a binary gammoid. Then $M \in \mathcal{G}_k$ for some $k$. By Theorem \ref{mt1}, $M$ contains $M(Q_i)$ minor, for $i=2,3,4$, a contradiction. Hence $M_H$ is a binary gammoid.
\end{proof}
\noindent In this section, we give alternate proof to the theorem given by Borse \cite{ymb_Gammoid}. We state the theorem again.\\
\begin{thm}\cite{ymb_Gammoid}\label{G2}
	Let a binary gammoid be $M$ then $M \in \mathcal{G}_2$ if and only if $M$ contain $M(G_1)$ minor, where graph $G_1$ is shown in the Figure \ref{G2fig}.
\end{thm}
\begin{proof}
	Let a binary gammoid $M$ contain minor $M(G_1)$. Then it is easy to prove that $M_{x,y}$ is not a binary gammoid, where $x,y$ are shown in Figure \ref{G2fig}.
	
	Conversely, let $M$ do not contain minor isomorphic to $M(G_1)$, then we prove that $M \notin \mathcal{G}_2$. Suppose, $M \in \mathcal{G}_2$ , that is $M_T$ is not a binary gammoid for some $T=\{x,y\} \subseteq E(M) $, then by Lemma \ref{gammoid}, $M_T$ contains minor $M(K_4)$. Then by Lemma \ref{minlem}, there is minor $P$ which is isomorphic to the one element extension of some minimal minor in  $\mathcal{G}_{k-1}$ or $P=M(Q_i)$ or $P$ is an extension of $M(Q_i)$ by one or two elements, where the graph $Q_i$ is as shown in Figure \ref{qk4fig}, for $i=2,3,4$. It is observed that the class $\mathcal{G}_1$ is empty. Thus $P \cong M(Q_i)$ or extension of $M(Q_i)$ not more than k elements, for $i=2,3,4$. By Theorem \ref{gammoid}, $P$ is a graphic gammoid thus $P\cong M(G)$. Let $G$ contains edge cut of two cardinality, say $\{a,b\}$. Then it is cocircuit of $P$ and $\{a,b\}$ contains cocircuit of $Q_i\backslash a$ for some $i$, which is contradiction. Hence $G$ does not contain edge cut containing two elements. \\
	Case-(i) If $P \cong M(Q_1)$, a contradiction, as $P$ is a minor of a binary gammoid. Hence we discard $Q_1$.
Case-(ii) If $P \cong M(Q_2)$ then $M(Q_2)_T$ will contain a loop or a pair of parallel arcs for any two-element set $T$. Hence $M(Q_2)_T \ncong M(K_4)$, thus graph $G$ is obtained by taking coextension of $Q_2$ by one or two elements such that it does not contain edge cut of two cardinalities. Extensions of $Q_2$ are given in the Figure \ref{extension}. Note that, $a_1 \cong G_1$. Hence we discard $a_1$. splitting $a_2$ and $a_3$ using two elements does not give $M(K_4)$; hence we take coextensions of $a_2$ and $a_3$ by one element such that it does not contain a cut set of two cardinalities.  

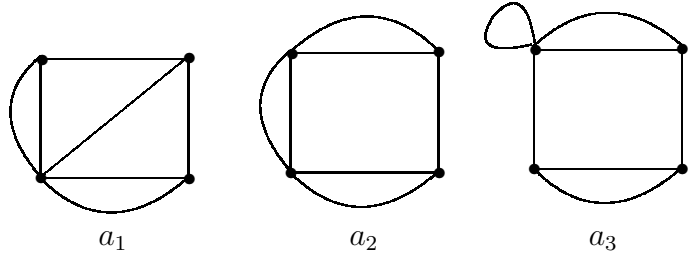
\begin{figure}[h!]
\centering
\unitlength 1mm 
\linethickness{0.4pt}
\ifx\plotpoint\undefined\newsavebox{\plotpoint}\fi 
\begin{picture}(94.435,35.75)(0,0)
	\put(28.81,9.25){\circle*{1.5}}
	\put(9.31,9.281){\circle*{1.5}}
	\put(9.158,9.299){\line(1,0){19.534}}
	\put(9.473,25){\circle*{1.5}}
	\put(28.79,25.25){\circle*{1.5}}
	\put(9.358,25.102){\line(1,0){19.325}}
	\put(28.683,25.102){\line(0,-1){15.608}}
	\put(9.21,9.196){\line(0,1){16.352}}
	\qbezier(9,9.5)(18.125,-0)(28.75,9.5)
	\qbezier(9.25,25.5)(1.375,18.875)(9,9.75)
	\multiput(28.75,25.5)(-.0409751037,-.0337136929){482}{\line(-1,0){.0409751037}}
	\put(61.685,10){\circle*{1.5}}
	\put(42.185,10.031){\circle*{1.5}}
	\put(42.033,10.049){\line(1,0){19.534}}
	\put(42.348,25.75){\circle*{1.5}}
	\put(61.665,26){\circle*{1.5}}
	\put(42.233,25.852){\line(1,0){19.325}}
	\put(61.558,25.852){\line(0,-1){15.608}}
	\put(42.085,9.946){\line(0,1){16.352}}
	\qbezier(41.875,10.25)(51,.75)(61.625,10.25)
	\qbezier(42.125,26.25)(34.25,19.625)(41.875,10.5)
	\put(93.685,10.5){\circle*{1.5}}
	\put(74.185,10.531){\circle*{1.5}}
	\put(74.033,10.549){\line(1,0){19.534}}
	\put(74.348,26.25){\circle*{1.5}}
	\put(93.665,26.5){\circle*{1.5}}
	\put(74.233,26.352){\line(1,0){19.325}}
	\put(93.558,26.352){\line(0,-1){15.608}}
	\put(74.085,10.446){\line(0,1){16.352}}
	\qbezier(73.875,10.75)(83,1.25)(93.625,10.75)
	\qbezier(42,26)(52.25,35.375)(61.5,26.25)
	\qbezier(74,26.75)(83.625,35.75)(93.75,26.75)
	\qbezier(73.75,27)(65.25,25.125)(68.75,30.75)
	\qbezier(68.75,30.75)(73,35.625)(74.25,27)
	\put(18.75,1){\makebox(0,0)[cc]{$a_1$}}
	\put(51.75,1){\makebox(0,0)[cc]{$a_2$}}
	\put(83.25,1){\makebox(0,0)[cc]{$a_3$}}
\end{picture}
\caption{Coextensions of $Q_2$ by one element}
\label{extension}
\end{figure}
Coextensions of $a_1$ are the graphs $b_1$ and $b_2$ and coextension of $a_3$ is $b_3$, where $b_1, b_2, b_3$ are given in Figure \ref{extension2}. Coextensions $b_1$, $b_2$ and $b_3$ contains minor $G_1$, hence we discard $b_1, b_2, b_3$ and hence $Q_2$.

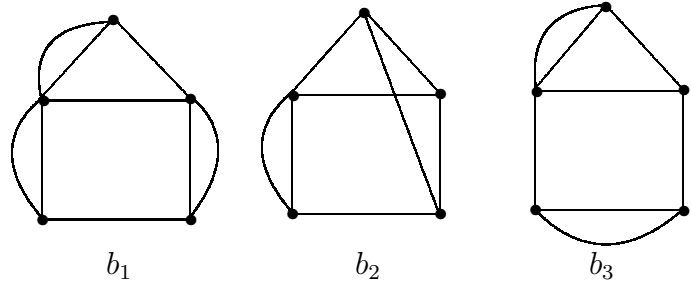
\begin{figure}[h!]
\centering
\unitlength 1mm 
\linethickness{0.4pt}
\ifx\plotpoint\undefined\newsavebox{\plotpoint}\fi 
\begin{picture}(99.06,36)(0,0)
	\put(33.435,7){\circle*{1.5}}
	\put(13.935,7.031){\circle*{1.5}}
	\put(13.783,7.049){\line(1,0){19.534}}
	\put(14.098,22.75){\circle*{1.5}}
	\put(33.415,23){\circle*{1.5}}
	\put(13.983,22.852){\line(1,0){19.325}}
	\put(33.308,22.852){\line(0,-1){15.608}}
	\put(13.835,6.946){\line(0,1){16.352}}
	\qbezier(13.875,23.25)(6,16.625)(13.625,7.5)
	\put(66.31,7.75){\circle*{1.5}}
	\put(46.81,7.781){\circle*{1.5}}
	\put(46.658,7.799){\line(1,0){19.534}}
	\put(46.973,23.5){\circle*{1.5}}
	\put(66.29,23.75){\circle*{1.5}}
	\put(46.858,23.602){\line(1,0){19.325}}
	\put(66.183,23.602){\line(0,-1){15.608}}
	\put(46.71,7.696){\line(0,1){16.352}}
	\qbezier(46.75,24)(38.875,17.375)(46.5,8.25)
	\put(98.31,8.25){\circle*{1.5}}
	\put(78.81,8.281){\circle*{1.5}}
	\put(78.658,8.299){\line(1,0){19.534}}
	\put(78.973,24){\circle*{1.5}}
	\put(98.29,24.25){\circle*{1.5}}
	\put(78.858,24.102){\line(1,0){19.325}}
	\put(98.183,24.102){\line(0,-1){15.608}}
	\put(78.71,8.196){\line(0,1){16.352}}
	\qbezier(78.5,8.5)(87.625,-1)(98.25,8.5)
	\put(23.25,33.5){\circle*{1.5}}
	\put(56.25,34.5){\circle*{1.5}}
	\put(88,35.25){\circle*{1.5}}
	\multiput(13.5,23)(.0337370242,.0371972318){289}{\line(0,1){.0371972318}}
	\multiput(23.25,33.75)(.0336700337,-.0353535354){297}{\line(0,-1){.0353535354}}
	\multiput(46.75,24)(.0336879433,.0381205674){282}{\line(0,1){.0381205674}}
	\multiput(56.25,34.75)(.0337370242,-.0371972318){289}{\line(0,-1){.0371972318}}
	\multiput(78.75,24.5)(.0336363636,.04){275}{\line(0,1){.04}}
	\multiput(88,35.5)(.0336700337,-.037037037){297}{\line(0,-1){.037037037}}
	\qbezier(13.75,23)(11.625,32.875)(23,33.25)
	\qbezier(33.25,23.5)(40.75,17.125)(33.25,7.25)
	\qbezier(78.75,24.75)(77.625,34.875)(88,35.5)
	\multiput(56,35)(.0336700337,-.0909090909){297}{\line(0,-1){.0909090909}}
	\put(24,1){\makebox(0,0)[cc]{$b_1$}}
	\put(56.75,1){\makebox(0,0)[cc]{$b_2$}}
	\put(87.5,1){\makebox(0,0)[cc]{$b_3$}}
\end{picture}

\caption{Coextensions of $a_1$ and $a_2$}
\label{extension2}
\end{figure}   
Case(iii) If $P \cong M(Q_3)$ or $P \cong M(Q_4)$ then using similar argument made above we can discard $Q_3$ and $Q_4$.  

Thus, from all the cases discuss above, we conclude that $M_T$ is a binary gammoid for any two element set $T \subseteq E(M)$, that is $M \notin \mathcal{G}_2$.
\end{proof}

\noindent We now find the minimal minors for the class $\mathcal{G}_3$. We state the theorem again.
\begin{thm}
	Let a binary gammoid be $M$ then $M \in \mathcal{G}_3$ if and only if $M$ contain $M(G_i)$ minor, where Figure \ref{G3}, shows graph $G_i$, for $i=2,3,4$.
\end{thm}
\begin{proof}
Let a binary gammoid $M$ contains minor $M(G_i)$, where $G_i$ is shown in Figure \ref{G3}, for $i=2,3,4$. Then we prove that $M \in \mathcal{G}_3$. \\
Let $H=\{x,y,z\}$ be as shown in graph $G_2$ in Figure \ref{G3} and $A_2$ be the  matrix representing the matroid $M(G_2)$ as given below.
	$ A_2 =\left[ \begin{array}{cccccc}
	x & y & z & & &  \\
	1 & 0 & 0 &1&0&1\\
	0 & 1 & 0 &0&1&1  \end{array} \right].$
 Then $ (A_2)_H =\left[ \begin{array}{cccccc}
x & y & z & & &  \\
1 & 0 & 0 &1&0&1\\
0 & 1 & 0 &0&1&1 \\
1 & 1 & 1 &0&0&0  \end{array} \right].$

It is observed that vector matroid $M(G_2)_H$ of the matrix $(A_2)_H$ is isomorphic to $M(K_4)$. Thus by Theorem \ref{gammoid}, $M(G_2)_H$ is not a binary gammoid. Similarly, we prove that, if $M$ contain minors $M(G_3)$ and $M(G_4)$ then $M_H$ is not a binary gammoid. Hence $M \in \mathcal{G}_3$.

Conversely, if $M$ does not contain any of $M(G_i)$ minor, for $i=2,3,4$. Then we prove that $M_H$ is a binary gammoid for any $H \subseteq E(M)$, with $|H|=3$. Suppose, $M_H$ is not a binary gammoid for some subset $H$ of $E(M)$. Then $M_H$ has a minor $M(K_4)$. Thus by lemma \ref{minlem}, there exists a minor $P$ of $M$, such that $P_H\cong M(K_4)$ or $P_H /H' \cong M(K_4)$ for some subset $H'$ of $H$. By Lemma \ref{rel_min_quo}, there is a binary gammoid $N$ with $a \in E(N)$, such that $N\backslash a \cong M(K_4)$ and $P \cong N/a$ or $P$ is extension of $N/a$ not more than three elements. By Lemma \ref{qk4}, $N/a \cong M(Q_i)$ for $i=1,2,3,4$. We discard the case of $M(Q_1)$ as it is not a binary gammoid. Note that, for $i=2,3,4$, $M(Q_i)\cong M(G_i)$, that is $M$ contains minor $M(G_i)$ for $i=2,3,4$, a contradiction. Hence $M_H$ is a binary gammoid for any set $H$ containing three elements. Hence proof.

\end{proof}

\section{Single Element Coextension of Gammoid}
From the definition of element splitting for a binary matroid $M$, it is clear that the single element coextension of $M$ is the element splitting matroid $M_H'$ for some $H \subseteq E(M)$. We have characterized a binary gammoid whose single element coextension is a binary gammoid. We state the theorem again.
\begin{thm}\label{mtelsp}
	Let $M$ be a binary gammoid then element splitting matroid $M_H'$ is a binary gammoid, for $H\subseteq E(M)$, with $|H|\geq 3$ if and only if $M$ do not contain $M(G_6)$ minor, where $G_6$ is as shown in Figure \ref{elementsp}.
\end{thm}
\begin{figure}[h]
\unitlength 1mm 
\linethickness{0.4pt}
\ifx\plotpoint\undefined\newsavebox{\plotpoint}\fi 
\begin{picture}(28.264,24.188)(0,0)
	\put(26.81,8.438){\circle*{1.5}}
	\put(16.81,23.438){\circle*{1.5}}
	\put(7.31,8.469){\circle*{1.5}}
	\put(7.158,8.487){\line(1,0){19.534}}
	\multiput(16.595,23.588)(.0336731392,-.0485889968){309}{\line(0,-1){.0485889968}}
	\multiput(16.743,23.291)(-.0336655052,-.0512787456){287}{\line(0,-1){.0512787456}}
	\qbezier(16.892,23.44)(28.264,21.135)(26.852,8.426)
	\qbezier(7.23,8.574)(5,20.764)(16.446,23.44)
	\put(17.25,3.25){\makebox(0,0)[cc]{$G_6$}}
	\put(6.25,19.25){\makebox(0,0)[cc]{$x$}}
	\put(28.25,19){\makebox(0,0)[cc]{$y$}}
	\put(17.25,10.25){\makebox(0,0)[cc]{$z$}}
\end{picture}
\caption{Forbidden minor for single element coextension of a binary gammoid}
\label{elementsp}
\end{figure}
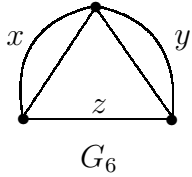
\begin{proof}
Let $M$ be a binary gammoid such that it contains a minor $M(G_6)$, then it is straightforward to prove that $M_H'$ is not a binary gammoid for $H=\{x,y,z\}$, where $x,y,z$ are shown in Figure \ref{elementsp}.\\
Conversely, let $M$ do not contain $M(G_6)$ minor then we need to prove that element splitting $M_H'$ is a binary gammoid. Suppose that $M_H'$ is not a binary gammoid, then by Theorem \ref{gammoid}, $M_H'$ contains a minor $M(K_4)$. Thus for some $H_1$ and $H_2$ in $E(M)$, $M_H'\backslash H_1 /H_2 \cong M(K_4)$. There are following cases.\\
Case-(i). If $a \notin H_1 \cup H_2$ then $M_H'\backslash H_1 /H_2/a \cong M(K_4)/a$ thus we have $M_H'/a \backslash H_1 /H_2 \cong M(G_6)$, but, $M_H'/a=M$ which gives that $M \backslash H_1 /H_2 \cong M(G_6)$, which is a contradiction. \\
Case-(ii). If $a \in H_1$ then, $M_H'\backslash a \backslash H_1-{a} /H_2 \cong M(K_4)$, but we have $M_H'\backslash a = M_H$. Thus we say that $M(K_4)$ is a minor of splitting matroid $M_H$. Then by Corollary \ref{maincor}, $M$ contains minor $M(Q_i)$, for $i=2,3,4$ and each $M(Q_i)$ contains a minor $M(G_6)$, which is a contradiction.\\
Case-(iii). If $a\in H_2$ then $M_H'/a \backslash H_1 /H_2-{a} \cong M \backslash H_1 /H_2-{a} \cong M(K_4)$, a contradiction. Thus from above cases we say that $M_H'$ is a binary gammoid.  
\end{proof}

\section{Es-splitting of a Binary Gammoid}
In this section, we characterize a binary gammoid whose es-splitting is a binary gammoid. We state the theorem below.
\begin{thm}
	Let $M$ be a binary gammoid then $M_H^e$ is a binary gammoid, for $H \subseteq E(M)$, with $|H|\geq 2$ and $e \in H$ if and only if $M$ do not contain a minor $M(G_7)$, where $G_7$ is as shown in Figure \ref{esgammoid}.
\end{thm}
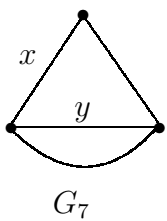
\begin{figure}[h]
\centering
\unitlength 1mm 
\linethickness{0.4pt}
\ifx\plotpoint\undefined\newsavebox{\plotpoint}\fi 
\begin{picture}(27.81,26.688)(0,0)
	\put(27.06,10.938){\circle*{1.5}}
	\put(17.06,25.938){\circle*{1.5}}
	\put(7.56,10.969){\circle*{1.5}}
	\put(7.408,10.987){\line(1,0){19.534}}
	\multiput(16.845,26.088)(.0336731392,-.0485889968){309}{\line(0,-1){.0485889968}}
	\multiput(16.993,25.791)(-.0336655052,-.0512787456){287}{\line(0,-1){.0512787456}}
	\put(15.5,.75){\makebox(0,0)[cc]{$G_7$}}
	\qbezier(26.75,11)(17.5,.5)(7.25,11)
	\put(9.75,20.25){\makebox(0,0)[cc]{$x$}}
	\put(17,13){\makebox(0,0)[cc]{$y$}}
\end{picture}

\caption{Forbidden minor for es-splitting of a binary gammoid}
\label{esgammoid}
\end{figure}
\begin{proof}
Let $M$ be a binary gammoid such that it contain $M(G_7)$ minor then it is straightforward to prove that $M_H^e$ is not a gammoid, for $H= \{x,y\}$ and $e \in H$.

Conversely, let the binary gammoid $M$ not contain $M(G_7)$ minor. Let $G$ be a graph corresponding to $M$, and $F$ be the graph obtained from $G$ by adding one parallel edge. Suppose that $M_H^e$ is not a binary gammoid. From the definition of es-splitting operation, it is clear that $M(F)_H'= M_H^e$ then $M(F)_H'$ is also not a binary gammoid. Thus, by Theorem \ref{mtelsp}, $F$ contain minor $M(G_6)$. Hence $M$ contains minor $M(G_7)$, a contradiction. Hence $M_H^e$ is a binary gammoid.
\end{proof}
\bibliographystyle{amsplain}

\begin{thebibliography}{10}
	\bibitem{azt} Azadi G., {\it Generalized splitting operation for
		binary matroids and related results}, Ph. D. Thesis, University of
	Pune (2001).
	\bibitem{azn} Azanchilar H. , Extension of line splitting operation from graphs to binary matroid, {\it Lobachevskii J. Math. 24} (2006), 3-12. 
	\bibitem{ymb_Gammoid} Borse Y. M., Forbidden-minors for splitting binary gammoid. {\it Austaralian Journal of Combinatorics} Vol. 46(2010), 307-314.
	\bibitem{Har} F. Harary, {\it Graph Theory, Narosa Publishing House, New Delhi} , 1988.
	\bibitem{gm} Ganesh Mundhe, {\it On Connectivity and Graphicness of Binary Matroids Under Splitting Operation}, Ph.D. Thesis, University of Pune, (2018).
	\bibitem{fl} H. Fleischner, {\it Eulerian Graphs and Related Topics Part 1, Vol. 1, North Holland, Amsterdam }, 1990.
		\bibitem{mms} M. M. Shikare and B. N. Waphare, Excluded-Minors for the class of graphic splitting matroids, {\it Ars
		Combin.}  97 (2010), 111-127.
	\bibitem{mms1}M. M. Shikare, Gh. Azadi, B. N. Waphare, Generalized splitting operation and its application, {\it J. Indian
		Math. Soc. }, 78, (2011), 145-154.

	\bibitem{ttr} T. T. Raghunathan, M. M. Shikare and B. N. Waphare, Splitting in a binary matroid, {\it Discrete Math. }
	184 (1998), 267-271.
	\bibitem{ox}J. G. Oxley, {\it Matroid Theory}, Oxford University Press, Oxford, 1992.
	\bibitem{ymb1} Y. M. Borse, M. M. Shikare and Pirouz Naiyer, A characterization of graphic matroids which yield
	cographic splitting matroids,{\it Ars Combin. }  118 (2015), 357-366.
	
\end{thebibliography}

\end{document}